\documentclass[10pt,a4paper]{amsart}
\usepackage[english]{babel}
\usepackage{amssymb,amsmath,amsthm}
\usepackage{enumitem}
\setlist[enumerate]{leftmargin=*}

\usepackage{hyperref}

\theoremstyle{plain}
\newtheorem{thm}{Theorem}[section]
\newtheorem{lem}[thm]{Lemma}

\newtheorem{prp}[thm]{Proposition}
\theoremstyle{definition}
\newtheorem{rem}[thm]{Remark}
\numberwithin{equation}{section}

\newcommand{\NN}{\mathbb{N}}
\newcommand{\ZZ}{\mathbb{Z}}
\newcommand{\RR}{\mathbb{R}}
\newcommand{\CC}{\mathbb{C}}

\newcommand{\Rpos}{{\RR^+}}
\newcommand{\Rnon}{{\RR^+_0}}

\newcommand{\CZ}{Calder\'on--Zygmund}
\newcommand{\CaC}{Carnot--Cara\-th\'eodory}

\newcommand{\ndi}{\mathrm{d}}  
\newcommand{\di}{\,\ndi}    

\newcommand{\fdim}{q}    
\newcommand{\vfG}{X}
\newcommand{\vfA}{\breve{X}_0}
\newcommand{\vfN}{\breve{X}}

\newcommand{\step}{S}

\newcommand{\lie}{\mathfrak}
\DeclareMathOperator{\tr}{tr}

\DeclareMathOperator{\arccosh}{arccosh}

\DeclareMathOperator{\Span}{span}

\newcommand{\dist}{\varrho}

\newcommand{\loc}{\mathrm{loc}}

\begin{document}

\title[Riesz transforms on solvable extensions of stratified groups]{Riesz transforms on solvable extensions \\ of stratified groups}

\author[A. Martini]{Alessio Martini}
\address[A. Martini]{School of Mathematics \\ University of Birmingham \\
Edgbaston \\ Birmingham \\ B15 2TT \\ United Kingdom}
\email{a.martini@bham.ac.uk}

\author[M. Vallarino]{Maria Vallarino}
\address[M. Vallarino]{Dipartimento di Scienze Matematiche ``Giuseppe Luigi Lagrange'', Dipartimento di Eccellenza 2018-2022
\\ Politecnico di Torino\\
Corso Duca degli Abruzzi 24\\ 10129 Torino\\ Italy}
\email{maria.vallarino@polito.it}

\subjclass[2010]{22E30, 42B20, 42B30}
\keywords{Riesz transform, sub-Laplacian, solvable group, singular integral operator, Hardy space, heat kernel}

\begin{abstract}
Let $G = N \rtimes A$, where $N$ is a stratified group and $A = \mathbb{R}$ acts on $N$ via automorphic dilations. Homogeneous sub-Laplacians on $N$ and $A$ can be lifted to left-invariant operators on $G$ and their sum is a sub-Laplacian $\Delta$ on $G$. Here we prove weak type $(1,1)$, $L^p$-boundedness for $p \in (1,2]$ and $H^1 \to L^1$ boundedness of the Riesz transforms $Y \Delta^{-1/2}$ and $Y \Delta^{-1} Z$, where $Y$ and $Z$ are any horizontal left-invariant vector fields on $G$, as well as the corresponding dual boundedness results. At the crux of the argument are large-time bounds for spatial derivatives of the heat kernel, which are new when $\Delta$ is not elliptic.
\end{abstract}

\maketitle

\section{Introduction}

Let $N$ be a stratified Lie group of homogeneous dimension $Q$. Let $G$ be the semidirect product $N \rtimes A$, where $A=\RR$ acts on $N$ via automorphic dilations. The group $G$ is a solvable extension of $N$ that is not unimodular and has exponential volume growth.
For all $p\in [1,\infty]$, let $L^p(G)$ denote the $L^p$ space with respect to a right Haar measure $\mu$ on $G$.

Consider a system $\vfN_1,\dots,\vfN_\fdim$ of left-invariant vector fields on $N$ that form a basis of the first layer of the Lie algebra of $N$ and let $\vfA$ be the standard basis of the Lie algebra of $A$. The vector fields $\vfA$ on $A$ and $\vfN_1,\dots,\vfN_\fdim$ on $N$ can be lifted to left-invariant vector fields $\vfG_0,\vfG_1,\dots,\vfG_\fdim$ on $G$ which generate the Lie algebra of $G$ and define a sub-Riemannian structure on $G$ with associated left-invariant \CaC\ distance $\dist$. A horizontal left-invariant  vector field on $G$ is any $\RR$-linear combination of $\vfG_0,\vfG_1,\dots,\vfG_\fdim$.

Let $\Delta$ be the left-invariant sub-Laplacian on $G$ defined by
\begin{equation}\label{eq:sub-Laplacian}
\Delta = -\sum_{j=0}^\fdim \vfG_j^2.
\end{equation}
The operator $\Delta$ extends uniquely to a positive self-adjoint operator on $L^2(G)$. Here is our main result.

\begin{thm}\label{thm:main}
Let $Y,Z$ be any horizontal left-invariant vector fields on $G$.
\begin{enumerate}[label=(\roman*)]
\item\label{en:main_first} The first-order Riesz transform $Y \Delta^{-1/2}$ is of weak type $(1,1)$, bounded on $L^p(G)$ for $p \in (1,2]$ and bounded from $H^1(G)$ to $L^1(G)$.
\item\label{en:main_second} The second-order Riesz transform $Y \Delta^{-1} Z$ is of weak type $(1,1)$, bounded on $L^p(G)$ for $p \in (1,\infty)$ and bounded from $H^1(G)$ to $L^1(G)$ and from $L^\infty(G)$ to $BMO(G)$.
\end{enumerate}
\end{thm}

We refer to Section \ref{ss:riesz_def} below for a precise definition of the Riesz transforms $Y \Delta^{-1/2}$ and $Y \Delta^{-1} Z$. The spaces $H^1(G)$ and $BMO(G)$ in the above statement are the Hardy and bounded mean oscillation spaces associated with the metric-measure space $(G,\dist,\mu)$ and its \CZ\ structure. It should be noted that the standard \CZ\ theory for doubling metric-measure spaces \cite{CW} does not apply to $(G,\dist,\mu)$, which has exponential volume growth. Instead we exploit the non-doubling \CZ\ theory of Hebisch and Steger \cite{HS} and the corresponding Hardy and BMO theory developed in \cite{MOV, V1,V3}.

Boundedness of Riesz transforms associated with Laplacians and sub-Laplacians has been studied in a variety of settings. Here we recall those results which have a direct connection with ours and refer to the cited works and references therein for a broader discussion. We remark that we are interested in $L^p$-boundedness with respect to a right Haar measure; in the case of a left Haar measure, the Riesz transforms $Y \Delta^{-1/2}$ considered here are not $L^p$-bounded for any $p \in [1,\infty]$ \cite{LM}.

In the case $N = \RR^Q$ is abelian, the operator $\Delta$ is a full Laplacian (although it is not the Laplace--Beltrami operator on the hyperbolic space $\RR^Q \rtimes \RR$). In this case, various parts of Theorem \ref{thm:main} are contained in already known results. Weak type $(1,1)$ and $L^p$-boundedness ($1 < p \leq 2$) of the first-order Riesz transforms $Y \Delta^{-1/2}$ is in \cite[Theorem 2.4]{HS} (previous partial results are in \cite{Sj}). Weak type $(1,1)$ and $L^p$-boundedness ($1 < p < \infty$) of the second-order Riesz transforms $Y \Delta^{-1} Z$ is in \cite{GQS,GS}. $H^1 \to L^1$ boundedness of first- and second-order Riesz transforms is studied in \cite{SV} in the particular case $Q=2$. The main novelty of our result lies in the fact that we can consider nonabelian $N$ and, correspondingly, nonelliptic $\Delta$. 

Other Riesz transforms on solvable extensions of stratified groups were previously studied in the literature for distinguished full Laplacians, especially in the context of 
 Iwasawa $NA$ groups of rank $1$ \cite{GQS, GS, W}. All these results involving a full Laplacian make strong use of spherical analysis on semisimple Lie groups (see also \cite{CGHM,HS,V1} for the study of spectral multipliers of a full Laplacian in such a context). This tool is not available for the analysis of the sub-Laplacian $\Delta$ on $G$ (unless $N$ is abelian), hence here different techniques are needed.

Due to the noncommutativity of $G$, the Riesz transform $Y \Delta^{-1} Z$ differs from $YZ \Delta^{-1}$ and $\Delta^{-1} YZ$. Indeed $YZ \Delta^{-1}$ and $\Delta^{-1} YZ$ are not $L^p$-bounded for any $p \in [1,\infty]$, at least when $N$ is abelian \cite{GQS,GS}, and therefore the Riesz transforms $Y \Delta^{-1} Z$ are the only ones for which it makes sense to investigate $L^p$-boundedness. Hence Theorem \ref{thm:main} gives a complete picture regarding $L^p$-boundedness of second-order Riesz transforms associated with $\Delta$, as well as $L^p$-boundedness for $p \leq 2$ for the first-order Riesz transforms $Y \Delta^{-1/2}$.

Note that the adjoint of $Y \Delta^{-1} Z$ is $Z \Delta^{-1} Y$, so it is natural that the $L^p$-boundedness range for these Riesz transforms is symmetric with respect to $p=2$.
The same does not hold for first-order transforms, and $L^p$-boundedness for $p>2$ of $Y \Delta^{-1/2}$ is equivalent to $L^p$-boundedness for $p < 2$ of $(Y \Delta^{-1/2})^* = - \Delta^{-1/2} Y$. For the operators $\Delta^{-1/2} Y$ very few results appear to be available in the literature: in the case $Q = 2$, it is known that the operators $\Delta^{-1/2} Y$ are not bounded from $H^1(G)$ to $L^1(G)$ \cite{SV}; on the other hand, in the case $Q = 1$ (i.e., the ``$ax+b$ group'' case), $\Delta^{-1/2} Y$ is known to be of weak type $(1,1)$ and $L^p$-bounded for $p \in (1,\infty)$ for a particular choice of $Y$ (i.e., when $Y \in \RR \vfG_1$) \cite{GS2}. While the existing results do not exclude that $\Delta^{-1/2} Y$ may be bounded for $p<2$ in greater generality, they seem to indicate that methods different from the ones employed in the present paper 
 would be necessary for such an investigation.

Indeed the proof of our results goes through showing that the Riesz transforms under consideration are singular integral operators of \CZ\ type, which implies weak type $(1,1)$ as well as $H^1 \to L^1$ boundedness.

As it is well-known, Riesz transforms can be subordinated to the heat semigroup $e^{-t\Delta}$ and their boundedness properties can be derived from estimates of suitable derivatives of the heat kernel $h_t$. It should be noted that, while small-time estimates for heat kernels associated to sub-Laplacians are available in great generality, precise large-time estimates on exponentially growing groups are known only in particular cases. Here we obtain weighted $L^1$-estimates
\[
\| e^{\epsilon |\cdot|_\dist/t^{1/2}} h_t \|_1 \lesssim_\epsilon 1, \qquad  \| e^{\epsilon |\cdot|_\dist/t^{1/2}} Y h_t \|_1 \lesssim_\epsilon t^{-1/2}
\]
for any $\epsilon \geq 0$ and all $t\in(0,\infty)$, where $|\cdot|_\dist$ denotes the $\dist$-distance from the identity and $Y$ is any horizontal left-invariant vector field. These estimates imply, via \CZ\ theory, the boundedness of the Riesz transforms $Y \Delta^{-1/2}$ for $p\in (1,2]$. The above heat kernel estimates extend those in \cite{HS} (where only the case $N$ abelian is considered and explicit formulas for the heat kernel are exploited, which are not available in our generality) and enhance those in \cite{MOV} (where unweighted estimates were proved for general $N$). 

As for the second-order Riesz transforms $Y \Delta^{-1} Z$, the argument for $N$ abelian in \cite{GQS,GS} is based, among other things, on a precise characterisation and asymptotic analysis of a fundamental solution of $\Delta^{-1}$, which is made possible by the fact that this fundamental solution is radial (after multiplication by a suitable power of the modular function).
These properties are no longer available in the case $N$ nonabelian and a different route is followed here, based on estimates of the second-order derivatives $Y (Z h_t)^*$ of the heat kernel (here $f \mapsto f^*$ is the usual $L^1$-isometric involution): note that $-\int_0^\infty Y (Z h_t)^* \di t$ is the convolution kernel of the Riesz transform $Y \Delta^{-1} Z$. Indeed the ``local part'' $\int_0^1 Y (Z h_t)^* \di t$ of the kernel is shown to satisfy estimates of \CZ\ type as before. The ``part at infinity'' $\int_1^\infty Y (Z h_t)^* \di t$, instead, turns out to be integrable. More precisely, in the case $Y, Z \in \Span\{\vfG_1,\dots,\vfG_\fdim\}$, we can prove that 
\[
\| e^{\epsilon |\cdot|_\dist/t^{1/2}} \, Y (Z h_t)^* \|_1 \lesssim_\epsilon \min \{t^{-1}, t^{-3/2} \}
\]
and the ``extra decay'' for large $t$ yields $\int_1^\infty Y (Z h_t)^* \di t \in L^1(G)$.
In the case one of $Y$ and $Z$ (or both) is a multiple of $\vfG_0$, instead, a more careful analysis is employed, exploiting additional cancellations occurring in the integration in $t$.

Our analysis is essentially based on a formula \cite{mustapha_multiplicateurs_1998,gnewuch_differentiable_2006} expressing the heat kernel $h_t$ on $G$ in terms of the corresponding heat kernel $h_t^N$ on $N$, as well as on a formula expressing the distance $\dist$ on $G$ in terms of the sub-Riemannian distance on $N$ \cite{H,MOV}. Through a number of manipulations, estimates for derivatives of $h_t$ are then reduced to estimates for derivatives of $h_t^N$, which are well-known. In these respects, our methods appear to be more robust than those used in previous works in the case $N$ abelian: instead of symmetry and radiality properties (which are not available for general $N$), here we directly exploit the semidirect product structure of $G = N \rtimes A$ to relate analysis on $G$ to analysis on $N$.

The $L^p$-boundedness of first-order Riesz transforms given in Theorem \ref{thm:main}\ref{en:main_first} might be the starting point to discuss properties of homogeneous Sobolev spaces defined in terms of the sub-Laplacian $\Delta$, in the spirit of the work done for nonhomogenous Sobolev spaces on nonunimodular groups in \cite{PV} and for both homegeneous and nonhomogeneous Sobolev spaces on unimodular groups in \cite{CRTN}. 

Moreover, it would be interesting to investigate boundedness properties for Riesz transforms associated with the sub-Laplacian with drift $\Delta-X$, where $X$ is a suitable horizontal left-invariant vector field, for which a multiplier theorem was proved in \cite{MOV}. Some results in this direction were obtained in \cite{LM2}.

\bigskip

The structure of the paper is as follows. In Section \ref{s:prelim} we recall a number of known facts regarding the groups $G$ and $N$, the sub-Riemannian structure, and the related \CZ\ theory. 
In Section \ref{s:heat} we derive weighted $L^1$-estimates for certain derivatives of the heat kernel. Finally, in Section \ref{s:riesz} we prove our main result, Theorem \ref{thm:main}.

\bigskip

Let us fix some notation that will be used throughout.
$\Rpos$ and $\Rnon$ denote the open and closed positive half-lines in $\RR$ respectively.
The letter $C$ and variants such as $C_s$ denote finite positive constants that may vary from place to place.
Given two expressions $A$ and $B$, $A\lesssim B$ means that there exists a finite positive constant $C$ such that $ A\le C \,B $. Moreover $A \sim B$ means $A \lesssim B$ and $B \lesssim A$.

\section{Preliminaries}\label{s:prelim}

The material presented in this section summarises a number of definitions and results extensively discussed in \cite{MOV}, to which we refer for details and references to the literature.

\subsection{Stratified groups and their extensions}\label{subs:NA}

Let $N$ be a stratified group. In other words, $N$ is a simply connected Lie group, whose Lie algebra $\lie{n}$ is equipped with a derivation $D$ such that the eigenspace of $D$ corresponding to the eigenvalue $1$ generates $\lie{n}$ as a Lie algebra. The eigenvalues of $D$ are positive integers $1,\dots,\step$ and $\lie{n}$ is the direct sum of the eigenspaces of $D$, which are called layers: the $j$th layer corresponds to the eigenvalue $j$. Moreover $\lie{n}$ is $\step$-step nilpotent, where $\step$ is the maximum eigenvalue.

The exponential map $\exp_N : \lie{n} \to N$ is a diffeomorphism and provides global coordinates for $N$, that shall be used in the sequel without further mention.
 Any chosen Lebesgue measure on $\lie{n}$ is then a left and right Haar measure on $N$, which we fix throughout. The formula $\delta_t = \exp((\log t) D)$ defines a family of automorphic dilations $(\delta_t)_{t>0}$ on $N$, and $Q = \tr D$ is called the homogeneous dimension of $N$. 

Let $A = \RR$, considered as an abelian Lie group. Again we identify $A$ with its Lie algebra $\lie{a}$. Then $A$ acts on $N$ by automorphic dilations, 
and we can define the corresponding semidirect product $G = N \rtimes A$, with operations
\begin{equation}\label{eq:group_law}
(z,u) \cdot (z',u') = (z \cdot e^{uD} z', u+u'), \qquad (z,u)^{-1} = (-e^{-uD} z,-u)
\end{equation}
and identity element $0_G=(0_N,0)$. $G$ is a solvable Lie group, and the Lie algebra $\lie{g}$ of $G$ is natually identified with the semidirect product of Lie algebras $\lie{n} \rtimes \lie{a}$.

The left and right Haar measures $\mu_\ell$ and $\mu$ on $G$ are given by
\[
\di\mu_\ell(z,u) = e^{-Qu} \di z \di u  \qquad \di\mu(z,u) = \di z \di u
\]
and the modular function $m$ is given by $m(z,u) =  e^{-Qu}$. In particular $G$ is not unimodular and has exponential volume growth. In the following the right Haar measure $\mu$ will be used to define Lebesgue spaces $L^p(G) = L^p(G,\di \mu)$ on $G$ and $\|f\|_p$ will denote the $L^p(G)$-norm of a function $f$ on $G$. Recall that $L^1(G)$ is a Banach $*$-algebra with respect to convolution and involution given by
\[
f * g(x) = \int_G f(xy^{-1}) \, g(y) \di \mu(y), \qquad f^*(x) = m(x) \overline{f(x^{-1})}
\]
for all $f,g \in L^1(G)$ and $x \in G$.

\subsection{Sub-Riemannian structure}\label{subs:metric}

Consider a system $\vfN_1,\dots,\vfN_\fdim$ of left-invariant vector fields on $N$ that form a basis of the first layer of $\lie{n}$. These vector fields provide a global frame for a sub-bundle $HN$ of the tangent bundle $TN$ of $N$, called the horizontal distribution. Since $N$ is stratified, the first layer generates $\lie{n}$ as a Lie algebra and consequently the horizontal distribution is bracket-generating.

Let $g^N$ be the left-invariant sub-Riemannian metric on the horizontal distribution of $N$ which makes $\vfN_1,\dots,\vfN_\fdim$ into an orthonormal basis, and $\dist^N$ the associated \CaC\ distance on $N$.
Since the horizontal distribution is bracket-generating, the distance $\dist^N$ is finite and induces on $N$ the usual topology.
Moreover, since $\vfN_1,\dots,\vfN_\fdim$ are left-invariant and belong to the first layer, the distance $\dist^N$ is left-invariant and homogeneous with respect to the automorphic dilations $\delta_t$.

Let $\vfA = \partial_u$ be the canonical basis of $\lie{a}$. The vector fields $\vfA$ on $A$ and $\vfN_1,\dots,\vfN_\fdim$ on $N$ can be lifted to left-invariant vector fields on $G$ given by
\[
\vfG_0|_{(z,u)} = \vfA|_z = \partial_u,  \qquad \vfG_j|_{(z,u)} = e^{u} \vfN_j|_z \qquad \text{ for $j=1,\dots,\fdim.$}
\]
Analogously as above, the system $\vfG_0,\dots,\vfG_\fdim$ generates the Lie algebra $\lie{g}$ and defines a sub-Riemannian structure on $G$ with associated horizontal distribution $HG$, sub-Riemannian metric $g$ and left-invariant \CaC\ distance $\dist$. 

Let $|x|_\dist = \dist(x,0_G)$ be the distance of $x \in G$ from the identity; similarly, let $|z|_N  = \dist^N(z,0_N)$ be the distance of $z \in N$ from the identity. The following relation between the \CaC\ distances on $G$ and $N$ is proved in \cite[Proposition 2.7]{MOV} (see also \cite{H}).

\begin{prp}\label{prp:distance}
For all $(z,u) \in G$,
\begin{equation}\label{eq:distanza}
|(z,u)|_\dist = \arccosh \left(\cosh u + e^{-u} |z|_N^2/2 \right).
\end{equation}
\end{prp}

For a (differentiable) function $f$ on $G$ we define the horizontal gradient $\nabla_H f(x) \in H_x G$ at $x \in G$ by
\[
g_x(\nabla_H f(x),v) = (\ndi f)_x(v)  \qquad\forall v \in H_x G,
\]
where $(\ndi f)_x$ is the differential of $f$ at $x$. 
 It is easily seen that
\begin{equation}\label{eq:nablaH}
|\nabla_H f(x)|^2_g  = g_x(\nabla_H f(x),\nabla_H f(x)) = \sum_{j=0}^\fdim |\vfG_j f(x)|^2.
\end{equation}
The close relation between the horizontal gradient $\nabla_H$ and the sub-Riemannian distance $\dist$ is clearly expressed by the following well-known estimate (see, e.g., \cite[VIII.1.1]{varopoulos_analysis_1992} and \cite[Lemma 5.4]{MOV}). Here we denote by $R_y$ the right translation operator defined by
\[
R_y f(x)=f(xy)
\]
for all $f : G \to \CC$ and $x,y\in G$.

\begin{lem}\label{lem:srgradient}
For all $f \in L^1_\loc(G)$ such that $\left|\nabla_{H} f\right|_g \in L^1(G)$, and for all $y,z \in G$,
\[
\|R_y f - R_z f\|_1 \leq \dist(y,z) \left\| \left|\nabla_{H} f\right|_g  \right\|_1.
\]
\end{lem}

\subsection{\CZ\ theory and Hardy spaces}\label{subs:abstractCZ}

A detailed description of the \CZ\ and Hardy space theory on $(G,\dist,\mu)$ can be found in \cite[Section 3]{MOV}, to which we refer also for the definition of the Hardy and bounded mean oscillation spaces $H^1(G)$ and $BMO(G)$. For the reader's convenience, here we record a criterion for boundedness of singular integral operators, which is a rephrasing of \cite[Theorem 1.2]{HS} and \cite[Theorems 3.2 and 3.8]{MOV} in the particular case of left-invariant operators on $L^2(G)$. Note that any such operator $T$ is a convolution operator:
\[
T \phi = \phi * k
\]
for some convolution kernel $k$ (which in general is a distribution on $G$) and all $\phi \in C_c^\infty(G)$; if $k$ is a locally integrable function, then
\[
T \phi(x) = \int K(x,y) \, \phi(y) \di\mu(y)
\]
for almost all $x \in G$, where the integral kernel $K$ of $T$ is given by
\[
K(x,y) =k(y^{-1}x)\,m(y) \qquad \text{for a.a. } x,y\in G.
\]

\begin{thm}\label{thm:Teolim}
Let $T$ be a linear operator bounded on $L^2(G)$ such that $T=\sum_{n\in \ZZ}T_n$, where
\begin{enumerate}[label=(\roman*)]
\item the series converges in the strong topology of operators on $L^2(G)$;
\item every $T_n$ is a left-invariant operator with convolution kernel $k_n \in L^1(G)$;
\item there exist positive constants $b,B,\varepsilon,c$ with $c \neq 1$ such that, for all $n \in \ZZ$,
\begin{gather}
\label{eq:czwl1} \int_G|k_n(x)|\,\big(1+c^n |x|_\dist \big)^{\varepsilon} \di\mu (x) \leq B,\\
\label{eq:czhoelder} \int_G|k_n^*(xy) - k_n^*(x)| \di\mu (x) \leq B\,\big(c^n |y|_\dist \big)^b\qquad\forall y\in G.
\end{gather}
\end{enumerate}
Then $T$ is of weak type $(1,1)$, bounded on $L^p(G)$ for $p \in (1,2]$, and bounded from $H^1(G)$ to $L^1(G)$.
\end{thm}

\begin{rem}\label{rem:Teolim_grad}
In view of Lemma \ref{lem:srgradient}, the condition \eqref{eq:czhoelder} with $b=1$ can be replaced by the stronger condition
\[
\int_G | \nabla_H k_n^*(x) |_g \di \mu(x) \leq B \,c^n.
\]
\end{rem}

\section{Heat kernel estimates for the \texorpdfstring{sub-Laplacian $\Delta$}{sub-Laplacian}}\label{s:heat}

\subsection{The sub-Laplacian and its heat kernel}
Let $\Delta$ be the sub-Laplacian defined in \eqref{eq:sub-Laplacian}. We now briefly recall some well-known properties of $\Delta$ and the associated heat kernel (see, e.g., \cite{varopoulos_analysis_1992} and \cite[Section 4.1]{MOV} for further details).

Since the horizontal distribution on $G$ is bracket-generating, $\Delta$ is hypoelliptic. Moreover $\Delta$ is essentially self-adjoint and positive with respect to the right Haar measure; in fact, for all $f,g \in C^\infty_c(G)$,
\begin{equation}\label{eq:dirichlet}
\langle \Delta f, g \rangle = \sum_{j=0}^\fdim \langle \vfG_j f , \vfG_j g \rangle,
\end{equation}
where $\langle \cdot,\cdot \rangle$ denotes the inner product of $L^2(G)$. In particular $\Delta$ extends uniquely to a positive self-adjoint operator on $L^2(G)$.

The heat kernel $t \mapsto h_t$ is a semigroup of probability measures on $G$. By hypoellipticity of $\partial_t + \Delta$, the distribution $(t,x) \mapsto h_t(x)$ is in fact a smooth function on $\Rpos \times G$ and satisfies
\[
h_t * h_{t'} = h_{t+t'}, \qquad h_t \geq 0, \qquad \| h_t \|_1 = 1,  \qquad h_t^* = h_t.
\]

It is also possible to obtain small-time ``Gaussian-type'' estimates for $h_t$ and its group-invariant derivatives. Let $\Sigma$ and $\Sigma_0$ denote the sets of finite sequences of elements of $\{1,\dots,\fdim\}$ and $\{0,\dots,\fdim\}$ respectively (note that $\Sigma \subseteq \Sigma_0$). For $\alpha = (j_1,\dots,j_k) \in \Sigma_0$ we write $|\alpha| = k$ and $\vfG^\alpha = \vfG_{j_1} \cdots \vfG_{j_k}$. Similarly one defines $\vfN^\alpha$ when $\alpha \in \Sigma$.

\begin{prp}\label{prp:small_time_gaussian_heat}
For all $x \in G$, $\alpha \in \Sigma_0$, and $t \in \Rpos$,
\[
|\vfG^\alpha h_t(x)| \leq C_\alpha \, t^{-(Q+1+|\alpha|)/2} e^{\omega_\alpha t} \exp(-b_\alpha |x|_\dist^2/t),
\]
where
$C_\alpha,b_\alpha,\omega_\alpha \in \Rpos$.
In particular, for all $t_0 \in \Rpos$, $\epsilon \in \Rnon$ and $\alpha,\beta \in \Sigma_0$,
\[
\| e^{\epsilon |\cdot|_\dist/t^{1/2}} \vfG^\alpha (\vfG^\beta h_t)^* \|_1 \leq C_{\alpha,\beta,t_0,\epsilon} \, t^{-(|\alpha|+|\beta|)/2}
\] 
for all $t \in (0,t_0]$.
\end{prp}
\begin{proof}
The pointwise estimate is a particular instance of \cite[Theorem 2.3(e)]{martini_spectral_2011}, which in turn is a rephrasing of results in \cite{ter_elst_weighted_1998}. Integrating this estimate against the weight $e^{\epsilon |\cdot|_\dist/t^{1/2}}$ on $G$ readily yields the $L^1$ estimate in the case $|\beta| = 0$ (see also the proof of \cite[Theorem 2.3(f)]{martini_spectral_2011}). Finally, the general case follows by observing that
\[
\vfG^\alpha (\vfG^\beta h_t)^* = \vfG^\alpha (\vfG^\beta (h_{t/2} * h_{t/2}))^* = (\vfG^\beta h_{t/2})^* * (\vfG^\alpha h_{t/2});
\]
hence, by Young's inequality,
\[
\|e^{\epsilon |\cdot|_\dist/t^{1/2}} \vfG^\alpha (\vfG^\beta h_t)^* \|_1 \leq \| e^{\epsilon |\cdot|_\dist/t^{1/2}} \vfG^\beta h_{t/2} \|_1 \, \| e^{\epsilon |\cdot|_\dist/t^{1/2}} \vfG^\alpha h_{t/2}\|_1
\]
and one can apply the particular case of the  	estimate to each factor.
\end{proof}

We remark that analogous estimates hold for general connected Lie groups and sub-Laplacians thereon,
but are effective for small times only; large-time estimates, instead, are a much more delicate problem in this generality.

One case where large-time estimates are not problematic is that of the heat kernel $h_t^N$ associated to the sub-Laplacian $\Delta^N=-\sum_{j=1}^\fdim \vfN_j^2$ on $N$. Here homogeneity considerations (cf.\ \cite[formula (1.73)]{folland_hardy_1982}) readily imply that
\begin{equation}\label{eq:n_heat_homogeneity}
\vfN^\alpha (\vfN^\beta h_{\lambda^{-2} t}^N)^*(z) = \lambda^{Q+|\alpha|+|\beta|} \vfN^\alpha (\vfN^\beta h_t^N)^*(\delta_{\lambda} z) \qquad \forall \lambda,t \in \Rpos, z\in N, \alpha,\beta \in \Sigma,
\end{equation}
whence precise weighted $L^1$-estimates follow.

\begin{prp}\label{prp:n_heat_estimates}
For all $\alpha,\beta \in \Sigma$ and $\gamma_0 \in \Rnon$, there exists $C_{\alpha,\beta,\gamma_0} \in \Rpos$ such that, for all $\gamma \in [0,\gamma_0]$ and all $s \in \Rpos$,
\[
\Bigl\||\cdot|_N^{2\gamma} \, \vfN^\alpha (\vfN^\beta h_s^N)^* \Bigr\|_{L^1(N)} \leq C_{\alpha,\beta,\gamma_0} s^{\gamma-(|\alpha|+|\beta|)/2}.
\]
\end{prp}

\subsection{Weighted \texorpdfstring{$L^1$}{L1}-estimates of heat kernel derivatives}\label{subs:heat}

The heat kernel $h_t$ associated to $\Delta$ can be expressed in terms of the heat kernel $h_t^N$ associated to the sub-Laplacian $\Delta^N$ 
(see \cite[\S 3]{mustapha_multiplicateurs_1998} or \cite[\S 2]{gnewuch_differentiable_2006}):
\begin{equation}\label{eq:integral}
h_t(z,u) = \int_0^\infty \Psi_t(\xi) \, \exp\left(-\frac{\cosh u}{\xi}\right) h^N_{e^u \xi /2}(z) \di\xi,
\end{equation}
where
\begin{equation}\label{eq:defPsi}
\Psi_t(\xi) = \frac{\xi^{-2}}{\sqrt{4\pi^3 t}} \exp\left(\frac{\pi^2}{4t} \right) \int_0^\infty \sinh\theta \, \sin\frac{\pi\theta}{2t} \,\exp\left(-\frac{\theta^2}{4t} -\frac{\cosh \theta}{\xi}\right) \di\theta.
\end{equation}

The following identities and estimates involving the first derivative of the function $\xi\mapsto \xi \, \Psi_t(\xi)$ will be useful in the sequel. 

\begin{lem}\label{lem:cancellation}
We have the following:
\begin{enumerate}[label=(\roman*)]
\item\label{en:cancellation_der} $\partial_\xi [\xi \, \Psi_t(\xi)] 
= \frac{\xi^{-2}}{4\sqrt{\pi^3 t^3}} \int_0^\infty \cosh\theta \, \left[\pi \cos\frac{\pi\theta}{2t} -  \theta \sin\frac{\pi\theta}{2t}\right] 
\exp\left(\frac{\pi^2-\theta^2}{4t}-\frac{\cosh \theta}{\xi}\right) \di\theta$;
\item\label{en:cancellation_derint} $\int_1^\infty \partial_\xi [\xi \Psi_t(\xi)] \di t = \frac{\xi^{-2}}{\sqrt{4\pi^3}} \int_0^\infty \int_0^\pi \cos \frac{s\theta}{2} \exp\left(\frac{s^2 -\theta^2}{4}-\frac{\cosh \theta}{\xi}\right) \cosh\theta \di s \di\theta$;
\item\label{en:cancellation_derintest} $\left|\int_1^\infty \partial_\xi [\xi \Psi_t(\xi)] \di t\right| \leq C \xi^{-2} \int_0^\infty \exp\left(-\frac{\theta^2}{4}-\frac{\cosh \theta}{\xi}\right) \cosh\theta \di\theta$.
\end{enumerate}
\end{lem}
\begin{proof}
The proof of \ref{en:cancellation_der} is given in \cite[Proposition 4.2]{MOV}. The identity \ref{en:cancellation_derint} follows from \ref{en:cancellation_der} and  
\[
\int_1^\infty  \exp\left(\frac{\pi^2 - \theta^2}{4t}\right) \left[ \pi \cos \frac{\pi\theta}{2t} - \theta \sin \frac{\pi\theta}{2t} \right] \frac{\ndi t}{t^{3/2}} = 2 \int_0^\pi \exp \left( \frac{s^2-\theta^2}{4} \right) \cos\frac{s\theta}{2} \di s,
\]
which we now prove. Note first that
\[
\frac{1}{t^{3/2}} \exp\left(\frac{\pi^2 - \theta^2}{4t}\right) \left[ \pi \cos \frac{\pi\theta}{2t} - \theta \sin \frac{\pi\theta}{2t} \right] = \Re \left[ \frac{\pi+i\theta}{t^{3/2}} \exp \left( \frac{(\pi+i\theta)^2}{4t} \right)\right]
\]
and
\[
\int_1^\infty \frac{\pi+i\theta}{t^{3/2}} \exp \left( \frac{(\pi+i\theta)^2}{4t} \right) \di t 
= 2 \int_0^{\pi+i\theta} \exp \left( \frac{w^2}{4} \right) \di w,
\]
where the latter is meant as a contour integral in the complex plane. Moreover,
\[
\int_0^{\pi+i\theta} \exp \left( \frac{w^2}{4} \right) \di w = \int_0^{i\theta} + \int_{i\theta}^{\pi + i\theta}
\]
and the first summand is purely imaginary, as it is easily seen by integrating along a straight line. In conclusion,
\begin{multline*}
\int_1^\infty \exp\left(\frac{\pi^2 - \theta^2}{4t}\right) \left[ \pi \cos \frac{\pi\theta}{2t} - \theta \sin \frac{\pi\theta}{2t} \right] \frac{\ndi t}{t^{3/2}} \\
= 2 \Re \int_{i\theta}^{\pi+i\theta} \exp \left( \frac{w^2}{4} \right) \di w = 2 \int_{0}^{\pi} \Re \exp \left( \frac{(s+i\theta)^2}{4} \right) \di s
\end{multline*}
and \ref{en:cancellation_derint} is proved.

The estimate \ref{en:cancellation_derintest} is an immediate consequence of \ref{en:cancellation_derint}. 
\end{proof}

The following technical lemma, which we will repeatedly use in the sequel, extends \cite[Lemma 4.1]{MOV} and can be proved following the same lines.

\begin{lem}\label{lem:innerintegral}
For all $\alpha,\beta,\theta \in \Rnon$ with $\beta \geq \alpha$, and $\delta \in [0,1/2]$,
\begin{multline*}
\int_\RR \int_0^\infty \frac{\cosh(\alpha u)}{\xi^{2+\beta}} \,  (\xi^\delta + (\cosh u)^\delta) \, \exp\left(-\frac{\cosh \theta+\cosh u}{\xi}\right) \di\xi \di u \\
\leq C_{\alpha,\beta} \begin{cases}
e^{(\delta-1 -\beta+\alpha)\theta} &\text{if $\alpha > 0$,}\\
e^{(\delta-1 -\beta)\theta} (1+\theta) &\text{if $\alpha = 0$,}
\end{cases}
\end{multline*}
where the constant $C_{\alpha,\beta}$ does not depend on $\delta$.
\end{lem}

We now derive weighted $L^1$-estimates for $h_t$ and its horizontal gradient, whose unweighted version is proved in \cite[Proposition 4.2]{MOV}. 

\begin{prp}\label{prp:gradientestimates}
For all $\epsilon \in \Rnon$, there exists $C_\epsilon \in \Rpos$ such that, for all $t \in \Rpos$,
\[
\left\| e^{\epsilon |\cdot|_\dist / t^{1/2}} \, h_t \right\|_{1} \leq C_\epsilon, \qquad \left\| e^{\epsilon |\cdot|_\dist / t^{1/2}} \, \left|\nabla_H h_t\right|_g\right\|_{1} \leq C_\epsilon \, t^{-1/2}.
\]
 \end{prp}
\begin{proof}
Fix $t_0 \in \Rpos$ sufficiently large so that $\epsilon/t_0^{1/2} \leq 1/2$.

By \eqref{eq:nablaH} it suffices to show that
\[
\|e^{\epsilon |\cdot|_\dist/t^{1/2}} \, h_t\|_{1} \leq C_\epsilon, \qquad \|e^{\epsilon |\cdot|_\dist/t^{1/2}} \, \vfG_j h_t\|_{1} \leq C_\epsilon \, t^{-1/2}
\]
for all $j \in \{0,1,\dots,\fdim\}$ and $t \in \Rpos$. In the case $t \leq t_0$, these estimates are given by Proposition \ref{prp:small_time_gaussian_heat}. Therefore in the rest of the proof we will assume that $t \geq t_0$. We remark that the constants in these estimates may depend on $\epsilon$, hence on $t_0$; therefore the same dependence is allowed for all the implicit constants in the estimates throughout the proof.

If $\gamma_t = \epsilon/t^{1/2}$ and $t \geq t_0$, then $\gamma_t \in [0,1/2]$ and therefore, for all $(z,u) \in G$, by Proposition \ref{prp:distance},
\begin{equation}\label{eq:expweight_dec}
\exp( \gamma_t |(z,u)|_\dist ) \leq (\cosh |(z,u)|_\dist)^{\gamma_t} \lesssim (\cosh u)^{\gamma_t} + (e^{-u} |z|_N^2)^{\gamma_t},
\end{equation}
where the implicit constant does not depend on $t \in [t_0,\infty)$.

We discuss first the estimate for $\vfG_j h_t$ in the case $j>0$. Recall that $\vfG_j = e^{u} \vfN_j$. Then, by \eqref{eq:integral} and differentiation under the integral sign,
\begin{equation}\label{eq:der_heat}
\vfG_j h_t(z,u) = \int_0^\infty \Psi_t(\xi) \, \exp\left(-\frac{\cosh u}{\xi}\right) e^{u} \vfN_j h^N_{e^u \xi /2}(z) \di\xi.
\end{equation}
Therefore, by \eqref{eq:expweight_dec} and Proposition \ref{prp:n_heat_estimates},
\[\begin{split}
\|e^{\gamma_t |\cdot|_\dist} &\,\vfG_j h_t\|_1 \\
&\lesssim \int_\RR \int_0^\infty (\cosh u)^{\gamma_t} \Bigl\| \vfN_j h^N_{e^u \xi /2}\Bigr\|_{L^1(N)}  |\Psi_t(\xi)| \, \exp\left(-\frac{\cosh u}{\xi}\right)  e^u \di\xi \di u  \\
&+ \int_\RR \int_0^\infty e^{-\gamma_t u} \Bigl\| |\cdot|_N^{2\gamma_t} \, \vfN_j h^N_{e^u \xi /2}\Bigr\|_{L^1(N)}  |\Psi_t(\xi)| \, \exp\left(-\frac{\cosh u}{\xi}\right)  e^u \di\xi \di u \\
&\lesssim \int_\RR \int_0^\infty |\Psi_t(\xi)| \, \frac{e^{u/2}}{\xi^{1/2}} ((\cosh u)^{\gamma_t} + \xi^{\gamma_t})  \exp\left(-\frac{\cosh u}{\xi}\right) \di\xi \di u .
\end{split}\]
Since $t\geq t_0$, by \eqref{eq:defPsi} the above integral is controlled by a constant (depending on $t_0$, but not on $t$) times
\begin{multline*}
t^{-1/2} \int_0^\infty \sinh\theta \, \left|\sin\frac{\pi\theta}{2t} \right| \,\exp\left(-\frac{\theta^2}{4t}\right) \\
\times \int_\RR \int_0^\infty \frac{e^{u/2}}{\xi^{2+1/2}} ((\cosh u)^{\gamma_t} + \xi^{\gamma_t}) \exp\left(-\frac{\cosh \theta+\cosh u}{\xi}\right) \di\xi \di u \di\theta.
\end{multline*}
By applying Lemma \ref{lem:innerintegral} (with $\alpha=\beta=1/2$), the integral in $u$ and $\xi$ is controlled by a constant times $e^{(\gamma_t-1)\theta}$, hence
\[
\|e^{\gamma_t |\cdot|_\dist} \,\vfG_j h_t\|_1 \lesssim t^{-1/2} \int_0^\infty \frac{\sinh\theta}{e^\theta} \, \frac{\theta}{t}  \,\exp\left(\frac{\epsilon \theta}{t^{1/2}}-\frac{\theta^2}{4t}\right) \di \theta \lesssim t^{-1/2}.
\]
This proves the estimate for $\vfG_j h_t$ in the case $j>0$. A similar argument, using Lemma \ref{lem:innerintegral} with $\alpha=\beta=0$, gives the estimate for $h_t$.

It remains to discuss the estimate for $\vfG_0 h_t$. Note that, again by \eqref{eq:integral},
\begin{equation}\label{eq:heat_X0der}
\begin{split}
\vfG_0 h_t(z,u) 
&= -\int_0^\infty \Psi_t(\xi) \, \frac{\sinh u}{\xi} \, \exp\left(-\frac{\cosh u}{\xi}\right) h^N_{e^u \xi /2}(z) \di \xi \\
&\qquad + \int_0^\infty \Psi_t(\xi) \, \exp\left(-\frac{\cosh u}{\xi}\right) \frac{\partial}{\partial u} [h^N_{e^u \xi /2}(z)] \di \xi \\
&= -\int_0^\infty \Psi_t(\xi) \, \frac{e^u}{\xi} \, \exp\left(-\frac{\cosh u}{\xi}\right) h^N_{e^u \xi /2}(z) \di \xi \\
&\qquad -\int_0^\infty \partial_\xi [\xi \, \Psi_t(\xi)] \, \exp\left(-\frac{\cosh u}{\xi}\right) \, h^N_{e^u \xi /2}(z) \di \xi \\
&= I_1 + I_2 .
\end{split}
\end{equation}
Here the fact that $\partial_u [h^N_{e^u \xi /2}(z)] = \xi \partial_\xi [h^N_{e^u \xi /2}(z)]$, integration by parts and the identity $\sinh u + \cosh u = e^u$ were used.

The norm $\|e^{\gamma_t |\cdot|_\dist}\, I_1 \|_1$ of the summand $I_1$ can be controlled analogously as above (here Lemma \ref{lem:innerintegral} is applied with $\alpha=\beta=1$). As for $I_2$, we observe that, by \eqref{eq:expweight_dec} and Proposition \ref{prp:n_heat_estimates},
\[\begin{split}
\|e^{\gamma_t |\cdot|_\dist}\, &I_2 \|_1 \\
&\lesssim \int_\RR \int_0^\infty (\cosh u)^{\gamma_t} \bigl\| h^N_{e^u \xi /2}\bigr\|_{L^1(N)}  \left| \partial_\xi [\xi \, \Psi_t(\xi)]\right| \, \exp\left(-\frac{\cosh u}{\xi}\right)  \di\xi \di u  \\
&+ \int_\RR \int_0^\infty e^{-\gamma_t u} \Bigl\| |\cdot|_N^{2\gamma_t} \, h^N_{e^u \xi /2}\Bigr\|_{L^1(N)} \left| \partial_\xi [\xi \, \Psi_t(\xi)]\right| \, \exp\left(-\frac{\cosh u}{\xi}\right)  \di\xi \di u \\
&\lesssim \int_\RR \int_0^\infty  \left| \partial_\xi [\xi \, \Psi_t(\xi)]\right|   \,  ((\cosh u)^{\gamma_t} + \xi^{\gamma_t})  \exp\left(-\frac{\cosh u}{\xi}\right) \di\xi \di u.
\end{split}\]
Consequently, by Lemma \ref{lem:cancellation}\ref{en:cancellation_der},
since $t \geq t_0$, $\|e^{\gamma_t |\cdot|_\dist}\, I_2 \|_1$ is bounded by a constant (depending on $t_0$) times
\begin{multline*}
t^{-3/2} \int_0^\infty \cosh\theta \, \left|\pi \cos\frac{\pi\theta}{2t} -  \theta \sin\frac{\pi\theta}{2t}\right| \exp\left(-\frac{\theta^2}{4t}\right) \\
\times \int_\RR \int_0^\infty \xi^{-2} \, ((\cosh u)^{\gamma_t} + \xi^{\gamma_t}) \, \exp\left(-\frac{\cosh \theta + \cosh u}{\xi}\right) \di\xi \di u \di\theta.
\end{multline*}
By applying Lemma \ref{lem:innerintegral} (with $\alpha=\beta=0$), the integral in $\xi$ and $u$ is controlled by a constant times $e^{(\gamma_t-1)\theta} (1+\theta)$, hence
\[
\|e^{\gamma_t |\cdot|_\dist}\,I_2\|_{1} \lesssim t^{-3/2} \int_0^\infty \frac{\cosh\theta}{e^\theta} (1 + \theta^2/t) \, (1+\theta) \, \exp\left(\frac{\epsilon \theta}{t^{1/2}}-\frac{\theta^2}{4t}\right) \di\theta \lesssim t^{-1/2}
\]
and we are done.
\end{proof}

\begin{rem}
Simple modifications of the proof of Proposition \ref{prp:gradientestimates} yield that
\[
\left\| e^{\epsilon |\cdot|_\dist / t^{1/2}} \, \vfG^\alpha h_t \right\|_{1} \leq C_{\epsilon,\alpha} \max\{t^{-|\alpha|/2}, t^{-1/2}\}
\]
for all $\epsilon \in \Rnon$, $t \in \Rpos$ and $\alpha \in \Sigma$ with $|\alpha|>0$.
\end{rem}

Similar techniques as above yield estimates for certain second-order derivatives of the heat kernel, which show an ``extra decay'' for large time.

\begin{prp}\label{prp:gradientestimates_mixed}
For all $\epsilon \in \Rnon$, there exists $C_{\epsilon} \in \Rpos$ such that, for all $t \in \Rpos$ and $j,l=1,\dots,\fdim$,
\[
\left\| e^{\epsilon |\cdot|_\dist / t^{1/2}} \, \vfG_l (\vfG_j h_t)^* \right\|_{1} \leq C_{\epsilon} \, \min\{t^{-1}, t^{-3/2}\}.
\]
\end{prp}
\begin{proof}
Choose $t_0$ sufficiently large so that $\epsilon/t_0^{1/2} \leq 1/2$. For $t \leq t_0$ the desired estimate follows from Proposition \ref{prp:small_time_gaussian_heat}, hence in what follows we assume that $t \geq t_0$.

By \eqref{eq:group_law}, \eqref{eq:n_heat_homogeneity} and \eqref{eq:der_heat} we deduce that
\begin{equation}\label{eq:nder_heat}
\begin{split}
(\vfG_j h_t)^*&(z,u) \\
&= e^{-Qu} \int_0^\infty \Psi_t(\xi) \, \exp\left(-\frac{\cosh u}{\xi}\right) e^{-u} \vfN_j h^N_{e^{-u} \xi /2}(-e^{-uD} z) \di\xi \\
&= \int_0^\infty \Psi_t(\xi) \, \exp\left(-\frac{\cosh u}{\xi}\right) (\vfN_j h^N_{e^{u} \xi /2})^*(z) \di\xi.
\end{split}
\end{equation}
Hence, for all $l=1,\dots,\fdim$,
\begin{equation}\label{eq:nnder_heat}
\vfG_l (\vfG_j h_t)^*(z,u) = \int_0^\infty \Psi_t(\xi) \, \exp\left(-\frac{\cosh u}{\xi}\right) e^u \vfN_l (\vfN_j h^N_{e^{u} \xi /2})^*(z) \di\xi .
\end{equation}
By proceeding as in the proof of Proposition \ref{prp:gradientestimates} and applying Proposition \ref{prp:n_heat_estimates}, if we define $\gamma_t = \epsilon/t^{1/2}$, then we obtain that, for $t \geq t_0$,
\[
\begin{split}
\|e^{\gamma_t |\cdot|_\dist} &\,\vfG_l (\vfG_j h_t)^* \|_1 \\
&\lesssim \int_\RR \int_0^\infty |\Psi_t(\xi)| \, \frac{1}{\xi} ((\cosh u)^{\gamma_t} + \xi^{\gamma_t})  \exp\left(-\frac{\cosh u}{\xi}\right) \di\xi \di u \\
&\lesssim t^{-1/2} \int_0^\infty \sinh\theta \, \left|\sin\frac{\pi\theta}{2t} \right| \,\exp\left(-\frac{\theta^2}{4t}\right) \\
&\quad\times \int_\RR \int_0^\infty \frac{1}{\xi^{2+1}} ((\cosh u)^{\gamma_t} + \xi^{\gamma_t}) \exp\left(-\frac{\cosh \theta+\cosh u}{\xi}\right) \di\xi \di u \di\theta\\ 
&\lesssim  t^{-1/2} \int_0^\infty \frac{\sinh\theta}{e^{\theta}} \, e^{(\gamma_t-1) \theta} \, (1+\theta) \, \frac{\theta}{t} \,   \di \theta \lesssim t^{-3/2} ,
\end{split}
\]
where Lemma \ref{lem:innerintegral} (with $\alpha=0$ and $\beta=1$), Proposition \ref{prp:n_heat_estimates} and the fact that $\gamma_t \leq 1/2$ were used.
\end{proof}

An analogous decay for $t$ large can be obtained for certain third-order derivatives of the heat kernel.

\begin{prp}\label{prp:gradientestimates_mixedmixed}
For all $\epsilon \in\Rnon$, there exists $C_{\epsilon} \in \Rpos$ such that, for all $t \in \Rpos$ and $j=1,\dots,\fdim$ and $k,l=0,\dots,\fdim$ with $(k,l) \neq (0,0)$,
\[
\left\| e^{\epsilon |\cdot|_\dist / t^{1/2}} \, \vfG_l \vfG_k (\vfG_j h_t)^* \right\|_{1} \leq C_{\epsilon} \, t^{-3/2}.
\]
\end{prp}
\begin{proof}
Choose $t_0$ sufficiently large so that $\epsilon/t_0^{1/2} \leq 1/2$. For $t \leq t_0$ the desired estimate follows from Proposition \ref{prp:small_time_gaussian_heat}.
Consider now the case $t \geq t_0$.
From \eqref{eq:nder_heat} we deduce that, for all $k,l=1,\dots,\fdim$,
\[
\vfG_l \vfG_k (\vfG_j h_t)^*(z,u) = \int_0^\infty \Psi_t(\xi) \, \exp\left(-\frac{\cosh u}{\xi}\right) e^{2u} \vfN_l \vfN_k (\vfN_j h^N_{e^{u} \xi /2})^*(z) \di\xi .
\]
By proceeding as in the proof of Proposition \ref{prp:gradientestimates_mixed}, if we define $\gamma_t = \epsilon/t^{1/2}$, then we obtain that, for $t \geq t_0$,
\[
\begin{split}
\|e^{\gamma_t |\cdot|_\dist} &\,\vfG_l \vfG_k (\vfG_j h_t)^*\|_1 \\
&\lesssim \int_\RR \int_0^\infty |\Psi_t(\xi)| \, \frac{e^{u/2}}{\xi^{3/2}} ((\cosh u)^{\gamma_t} + \xi^{\gamma_t})  \exp\left(-\frac{\cosh u}{\xi}\right) \di\xi \di u \\
&\lesssim t^{-1/2} \int_0^\infty \sinh\theta \, \left|\sin\frac{\pi\theta}{2t} \right| \,\exp\left(-\frac{\theta^2}{4t}\right) \\
&\quad\times \int_\RR \int_0^\infty \frac{e^{u/2}}{\xi^{2+3/2}} ((\cosh u)^{\gamma_t} + \xi^{\gamma_t}) \exp\left(-\frac{\cosh \theta+\cosh u}{\xi}\right) \di\xi \di u \di\theta\\ 
&\lesssim  t^{-1/2} \int_0^\infty \frac{\sinh\theta}{e^{\theta}} \, e^{(\gamma_t-1) \theta}  \, \frac{\theta}{t} \,   \di \theta \lesssim t^{-3/2}
\end{split}
\]
where Lemma \ref{lem:innerintegral} (with $\alpha=1/2$ and $\beta=3/2$), Proposition \ref{prp:n_heat_estimates} and the fact that $\gamma_t \leq 1/2$ were used. This proves the desired estimate when $k \neq 0 \neq l$.

Consider now the case where $k=0 \neq l$. Starting from \eqref{eq:nder_heat} and proceeding as in the derivation of \eqref{eq:heat_X0der}, one easily obtains that 
\[
\vfG_l \vfG_0 (\vfG_j h_t)^*(z,u) = I_1 + I_2,
\]
where
\begin{align*}
I_1 &= -\int_0^\infty \Psi_t(\xi) \, \frac{e^u}{\xi} \exp\left(-\frac{\cosh u}{\xi}\right) \, e^u \vfN_l (\vfN_j h^N_{e^u \xi/2})^*(z) \di\xi, \\
I_2 &= -\int_0^\infty \partial_\xi [\xi \, \Psi_t(\xi)] \, \exp\left(-\frac{\cosh u}{\xi}\right) \, e^u \vfN_l (\vfN_j h^N_{e^u \xi /2})^*(z) \di\xi.
\end{align*}
The estimate $\|e^{\gamma_t |\cdot|_\dist} \, I_1\|_1 \lesssim t^{-3/2}$ for $t \geq t_0$ is then proved analogously as above (here Proposition \ref{prp:n_heat_estimates} and Lemma \ref{lem:innerintegral} with $\alpha=1$ and $\beta=2$ are used). As for $I_2$, from Lemma \ref{lem:cancellation}\ref{en:cancellation_der} and Proposition \ref{prp:n_heat_estimates} we obtain, for $t \geq t_0$,
\[\begin{split}
\|e^{\gamma_t |\cdot|_\dist}\, &I_2 \|_1 \\
&\lesssim \int_\RR \int_0^\infty  \left|\partial_\xi [\xi \, \Psi_t(\xi)]\right|  \frac{1}{\xi} \,  ((\cosh u)^{\gamma_t} + \xi^{\gamma_t})  \exp\left(-\frac{\cosh u}{\xi}\right) \di\xi \di u\\
&\lesssim t^{-3/2} \int_0^\infty \cosh\theta \, \left|\pi \cos\frac{\pi\theta}{2t} -  \theta \sin\frac{\pi\theta}{2t}\right| \exp\left(-\frac{\theta^2}{4t}\right) \\
&\quad\times \int_\RR \int_0^\infty \frac{1}{\xi^{2+1}} \, ((\cosh u)^{\gamma_t} + \xi^{\gamma_t}) \, \exp\left(-\frac{\cosh \theta + \cosh u}{\xi}\right) \di\xi \di u \di\theta \\
&\lesssim t^{-3/2} \int_0^\infty \frac{\cosh\theta}{e^\theta} e^{(\gamma_t -1)\theta} \, (1 + \theta)^2 \di\theta \lesssim t^{-3/2}
\end{split}\]
where Lemma \ref{lem:innerintegral} (with $\alpha=0$ and $\beta=1$) and the fact that $\gamma_t \leq 1/2$ were used.

As for the case where $l=0\neq k$, note that $[\vfG_0,\vfG_k] = \vfG_k$, hence
\[
\vfG_0 \vfG_k (\vfG_j h_t)^* = \vfG_k (\vfG_j h_t)^* + \vfG_k \vfG_0 (\vfG_j h_t)^*.
\]
The desired estimate then follows from the estimate $\|e^{\gamma_t |\cdot|_\dist} \, \vfG_k \vfG_0 (\vfG_j h_t)^*\|_1 \lesssim t^{-3/2}$ for $t \geq t_0$, which has just been proved, and the estimate $\|e^{\gamma_t |\cdot|_\dist} \, \vfG_k (\vfG_j h_t)^*\|_1 \lesssim t^{-3/2}$ for $t \geq t_0$, which is in Proposition \ref{prp:gradientestimates_mixed}.
\end{proof}

Finally, we obtain some pointwise estimates for second-order derivatives of the heat kernel for large time.
While they do not yield decay in the space variables, these estimates will be nevertheless useful to justify the absolute convergence of the integral defining the ``part at infinity'' of the kernel of the second-order Riesz transforms (see formula \eqref{eq:kinfty} below).

\begin{prp}\label{prp:pointwise_heat}
For all $t_0 \in \Rpos$, $l,j \in \{0,\dots,\fdim\}$, there exists $C_{t_0}\in\Rpos$ such that, for all $t \geq t_0$ and $(z,u) \in G$,
\[
|\vfG_j (\vfG_l h_t)^* (z,u)| \leq C_{t_0} \, t^{-3/2} e^{-Qu/2} \cosh u .
\]
\end{prp}
\begin{proof}
Let us preliminarily observe that from \eqref{eq:defPsi} and Lemma \ref{lem:cancellation}\ref{en:cancellation_der} one readily obtains that, for all $t \geq t_0$ and $\xi > 0$,
\begin{equation}\label{eq:est_Psi_der}
\begin{aligned}
|\Psi_t(\xi)| &\lesssim t^{-3/2} \int_0^\infty \theta \, \frac{\sinh \theta}{\xi^2} \exp\left(-\frac{\cosh \theta}{\xi} \right) \di\theta,\\
|\partial_\xi [ \xi \Psi_t(\xi)] | &\lesssim t^{-3/2} \int_0^\infty (1+\theta) \, \frac{\cosh \theta}{\xi^2} \exp\left(-\frac{\cosh \theta}{\xi} \right) \di\theta,\\
|\partial_\xi [\xi \partial_\xi [ \xi \Psi_t(\xi)]] | &\lesssim t^{-3/2} \int_0^\infty (1+\theta) \left[ \frac{\cosh \theta}{\xi^2} + \frac{\cosh^2 \theta}{\xi^3} \right] \exp\left(-\frac{\cosh \theta}{\xi} \right) \di\theta,
\end{aligned}
\end{equation}
where the implicit constants may depend on $t_0$.

Suppose first that $j,l \in \{1,\dots,q\}$. By \eqref{eq:nnder_heat}, \eqref{eq:n_heat_homogeneity} and \eqref{eq:est_Psi_der}, we deduce that
\begin{equation}\label{eq:pwest}
\begin{split}
|\vfG_j &(\vfG_l h_t)^* (z,u)| 
\leq \int_0^\infty |\Psi_t(\xi)| \, \exp\left(-\frac{\cosh u}{\xi}\right) e^u |\vfN_j (\vfN_l h^N_{e^{u} \xi /2})^*(z)| \di\xi \\
&\lesssim e^{-Qu/2} \int_0^\infty |\Psi_t(\xi)| \, \exp\left(-\frac{1}{\xi}\right)  \xi^{-Q/2-1} \di\xi \\
&\lesssim t^{-3/2} e^{-Qu/2} \int_0^\infty \theta \sinh\theta \int_0^\infty \exp\left(-\frac{\cosh \theta+1}{\xi}\right) \xi^{-Q/2-3} \di\xi \di\theta
\end{split}
\end{equation}
and the integral in $\theta$ and $\xi$ is immediately seen to be convergent.

Assume instead that $j \in \{1,\dots,q\}$ and $l = 0$. From \eqref{eq:heat_X0der}, \eqref{eq:group_law} and \eqref{eq:n_heat_homogeneity} it follows that
\begin{equation}\label{eq:heat_X0der_star}
\begin{split}
(\vfG_0 h_t)^*(z,u) 
&= -\int_0^\infty \Psi_t(\xi) \, \frac{e^{-u}}{\xi} \, \exp\left(-\frac{\cosh u}{\xi}\right) h^N_{e^u \xi /2}(z) \di \xi \\
&\qquad -\int_0^\infty \partial_\xi[\xi \, \Psi_t(\xi)] \, \exp\left(-\frac{\cosh u}{\xi}\right) \, h^N_{e^u \xi /2}(z) \di \xi,
\end{split}
\end{equation}
whence
\begin{equation}\label{eq:XjX0_dec}
\begin{split}
\vfG_j (\vfG_0 h_t)^*(z,u) 
&= -\int_0^\infty \Psi_t(\xi) \, \frac{1}{\xi} \, \exp\left(-\frac{\cosh u}{\xi}\right) \vfN_j h^N_{e^u \xi /2}(z) \di \xi \\
&\qquad -\int_0^\infty \partial_\xi[\xi \, \Psi_t(\xi)] \, \exp\left(-\frac{\cosh u}{\xi}\right) \, e^u \vfN_j h^N_{e^u \xi /2}(z) \di \xi \\
&= I_1 + I_2.
\end{split}
\end{equation}
Analogously as in \eqref{eq:pwest}, one can see that
\begin{equation}\begin{split}
|I_1| &\leq \int_0^\infty |\Psi_t(\xi)| \, \frac{1}{\xi} \, \exp\left(-\frac{\cosh u}{\xi}\right) |\vfN_j h^N_{e^u \xi /2}(z)| \di \xi \\
&\lesssim t^{-3/2} e^{-u(Q+1)/2}.
\end{split}\end{equation}
Moreover, by \eqref{eq:n_heat_homogeneity} and \eqref{eq:est_Psi_der},
\begin{equation}\label{eq:I2_absint}
\begin{split}
|I_2| &\leq \int_0^\infty |\partial_\xi[\xi \, \Psi_t(\xi)]| \, \exp\left(-\frac{\cosh u}{\xi}\right) \, e^u |\vfN_j h^N_{e^u \xi /2}(z)| \di \xi \\
&\lesssim t^{-3/2} e^{-u(Q-1)/2} \int_0^\infty (1+\theta) \cosh \theta \int_0^\infty \exp\left(-\frac{\cosh \theta +1}{\xi}\right) \, \xi^{-(Q+5)/2}  \di\xi \di\theta \\
&\lesssim t^{-3/2} e^{-u(Q-1)/2},
\end{split}
\end{equation}
since the integral in $\theta$ and $\xi$ is convergent. In conclusion
\[
|\vfG_j (\vfG_0 h_t)^*(z,u)| \lesssim t^{-3/2} e^{-Qu/2} \cosh (u/2).
\]

Note that, if $l \in \{1,\dots,q\}$ and $j = 0$, then $\vfG_0 (\vfG_l h_t)^* = (\vfG_l (\vfG_0 h_t)^*)^*$, hence again
\[
|\vfG_0 (\vfG_l h_t)^*(z,u)| \lesssim t^{-3/2} e^{-Qu/2} \cosh (u/2).
\]

Finally, assume that $j=l=0$. From \eqref{eq:heat_X0der_star}, by proceeding analogously as in the derivation of \eqref{eq:heat_X0der},
\begin{equation}\label{eq:X0X0_dec}
\begin{split}
\vfG_0 (\vfG_0 h_t)^*&(z,u) \\
&= \int_0^\infty \Psi_t(\xi) \, \frac{e^{-u}}{\xi} \, \exp\left(-\frac{\cosh u}{\xi}\right) h^N_{e^u \xi /2}(z) \di \xi \\
&\qquad +\int_0^\infty \Psi_t(\xi) \, \frac{e^{-u} \sinh u}{\xi^2} \, \exp\left(-\frac{\cosh u}{\xi}\right) h^N_{e^u \xi /2}(z) \di \xi \\
&\qquad -\int_0^\infty \Psi_t(\xi) \, \frac{e^{-u}}{\xi} \, \exp\left(-\frac{\cosh u}{\xi}\right) \xi \partial_\xi [h^N_{e^u \xi /2}(z)] \di \xi \\
&\qquad +\int_0^\infty \partial_\xi [\xi \, \Psi_t(\xi)] \, \frac{\sinh u}{\xi} \exp\left(-\frac{\cosh u}{\xi}\right) \, h^N_{e^u \xi /2}(z) \di \xi \\
&\qquad -\int_0^\infty \partial_\xi [\xi \, \Psi_t(\xi)] \, \exp\left(-\frac{\cosh u}{\xi}\right) \, \xi \partial_\xi [h^N_{e^u \xi /2}(z)] \di \xi \\
&=\int_0^\infty \Psi_t(\xi) \, \frac{1}{\xi^2} \, \exp\left(-\frac{\cosh u}{\xi}\right) h^N_{e^u \xi /2}(z) \di \xi \\
&\qquad +2 \int_0^\infty \partial_\xi [\xi \, \Psi_t(\xi)] \, \frac{\cosh u}{\xi} \exp\left(-\frac{\cosh u}{\xi}\right) \, h^N_{e^u \xi /2}(z) \di \xi \\
&\qquad +\int_0^\infty \partial_\xi [\xi \partial_\xi [\xi \, \Psi_t(\xi)] ] \, \exp\left(-\frac{\cosh u}{\xi}\right) \, h^N_{e^u \xi /2}(z) \di \xi \\
&= J_1 + 2 J_2 + J_3.
\end{split}
\end{equation}
Similarly as before, one then sees that
\begin{equation}\label{eq:J123_absint}
\begin{aligned}
|J_1| &\leq \int_0^\infty |\Psi_t(\xi)| \, \frac{1}{\xi^2} \, \exp\left(-\frac{\cosh u}{\xi}\right) |h^N_{e^u \xi /2}(z)| \di \xi \lesssim t^{-3/2} e^{-Qu/2}, \\
|J_2| &\leq \int_0^\infty |\partial_\xi [\xi \, \Psi_t(\xi)]| \, \frac{\cosh u}{\xi} \exp\left(-\frac{\cosh u}{\xi}\right) \, |h^N_{e^u \xi /2}(z)| \di \xi 
\lesssim t^{-3/2} \frac{\cosh u}{e^{Qu/2}} , \\
|J_3| &\leq \int_0^\infty |\partial_\xi [\xi \partial_\xi [\xi \, \Psi_t(\xi)] ]| \, \exp\left(-\frac{\cosh u}{\xi}\right) \, |h^N_{e^u \xi /2}(z)| \di \xi 
\lesssim t^{-3/2} e^{-Qu/2},
\end{aligned}
\end{equation}
so
\[
|\vfG_0 (\vfG_0 h_t)^*(z,u)| \lesssim t^{-3/2} e^{-Qu/2} \cosh u
\]
and we are done.
\end{proof}

\section{Riesz transforms}\label{s:riesz}

\subsection{Preliminaries on the definition of Riesz transforms}\label{ss:riesz_def}

Let $Y$ be a horizontal vector field. We consider $Y$ as a densely defined operator on $L^2(G)$, with maximal domain (i.e., the set of the $f \in L^2(G)$ whose distributional derivative $Y f$ is in $L^2(G)$). Then $i Y$ is self-adjoint and $C^\infty_c(G)$ is dense in the domain of $Y$ with respect to the graph norm (see, e.g., \cite{NS1959}). Similarly, the sub-Laplacian $\Delta$ is self-adjoint, and moreover it has trivial $L^2$ kernel, so via the spectral theorem we can define its fractional powers $\Delta^\alpha$ ($\alpha \in \RR$) as self-adjoint operators on $L^2(G)$.

\begin{prp}\label{prp:def_riesz}
Let $Y$ and $Z$ be horizontal left-invariant vector fields. 
\begin{enumerate}[label=(\roman*)]
\item\label{en:rieszdef_dom} The domain of $\Delta^{1/2}$ is contained in the domain of $Y$.
\item\label{en:rieszdef_first} The range of $\Delta^{-1/2}$ is contained in the domain of $Y$. The composition $Y \Delta^{-1/2}$, defined on the domain of $\Delta^{-1/2}$, extends to a bounded operator $\overline{Y \Delta^{-1/2}}$ on $L^2(G)$ with norm at most $|Y|_g$.
\item\label{en:rieszdef_adjfirst} The range of $Y$ is contained in the domain of $\Delta^{-1/2}$. The composition $\Delta^{-1/2} Y$, defined on the domain of $Y$, extends to a bounded operator $\overline {\Delta^{-1/2} Y}$ on $L^2(G)$ with norm at most $|Y|_g$.
\item\label{en:rieszdef_adjadj} $(\Delta^{-1/2} Y)^* = -\overline{Y \Delta^{-1/2}}$ and $(Y \Delta^{-1/2})^* = -\overline{\Delta^{-1/2} Y}$.
\item\label{en:rieszdef_firstappr} If $(F_n)_{n\in \NN}$ is an increasing sequence of nonnegative bounded Borel functions on $(0,\infty)$ converging pointwise to $\lambda \mapsto \lambda^{-1/2}$, then the range of $F_n(\Delta)$ is contained in the domain of $Y$, the operators $Y F_n(\Delta)$ and $\overline{F_n(\Delta) Y}$ are bounded on $L^2(G)$ with norm at most $|Y|_g$, and
\[
Y F_n(\Delta) \to \overline{Y \Delta^{-1/2}}, \qquad \overline{F_n(\Delta) Y} \to \overline{\Delta^{-1/2} Y}
\]
in the strong operator topology.
\item\label{en:rieszdef_secondappr} If $(G_n)_{n\in \NN}$ is an increasing sequence of nonnegative bounded Borel functions on $(0,\infty)$ converging pointwise to $\lambda \mapsto \lambda^{-1}$, then the range of $G_n(\Delta) Z$ is contained in the domain of $Y$, the operators $\overline{Y G_n(\Delta) Z}$ are bounded on $L^2(G)$ with norm at most $|Y|_g |Z|_g$, and
\[
\overline{Y G_n(\Delta) Z} \to \overline{Y \Delta^{-1/2}} \, \overline{\Delta^{-1/2} Z}
\]
in the strong operator topology.
\end{enumerate}
\end{prp}
\begin{proof}
Part \ref{en:rieszdef_dom} is a consequence of the fact that
\[
\|Y f\|_2 \leq |Y|_g \, \| |\nabla f|_g \|_2, \qquad \| |\nabla f|_g \|_2^2 = \langle \Delta f, f \rangle = \| \Delta^{1/2} f \|_2^2
\]
for all $f \in C^{\infty}_c(G)$, and the density of $C^\infty_c(G)$ in the domains of $Y$ and $\Delta^{1/2}$ for the respective graph norms.
From this part \ref{en:rieszdef_first} follows immediately. Similarly, if the sequence $(F_n)_{n \in \NN}$ is as in part \ref{en:rieszdef_firstappr} (note that such sequences do exist), then the range of $F_n(\Delta)$ is contained in the domain of $Y$ and the composition $Y F_n(\Delta)$ is bounded with norm at most $|Y|_g$. Note now that, for all $f$ in the domain of $\Delta^{-1/2}$,
\[
\|Y F_n(\Delta) f - Y \Delta^{-1/2} f\|_2 \leq |Y|_g \| \Delta^{1/2} (F_n(\Delta) - \Delta^{-1/2}) f\|_2
\]
and the right-hand side tends to $0$ as $n \to \infty$ by the properties of the Borel functional calculus; this and the uniform boundedness of the $Y F_n(\Delta)$ immediately imply that $Y F_n(\Delta) \to \overline{Y \Delta^{-1/2}}$ in the strong operator topology.

Note now that $F_n(\Delta) Y$ is densely defined and $(F_n(\Delta) Y)^* = - Y F_n(\Delta)$, whence $F_n(\Delta) Y$ is also bounded with norm $|Y|_g$. In addition $\|\Delta^{-1/2} f\|_2 = \sup_n \|F_n(\Delta) f\|_2$ for all $f \in L^2(G)$ (where $\|\Delta^{-1/2} f\|_2 = \infty$ if $f$ is not in the domain of $\Delta^{-1/2}$), so the uniform boundedness of the operators $F_n(\Delta) Y$ implies that the range of $Y$ is contained in the domain of $\Delta^{-1/2}$, and part \ref{en:rieszdef_adjfirst} follows.

Since both $\Delta^{-1/2} Y$ and $Y \Delta^{-1/2}$ are densely defined and bounded, part \ref{en:rieszdef_adjadj} follows immediately. Moreover, by functional calculus, $F_n(\Delta) Y f \to \Delta^{-1/2} Y f$ in $L^2$ for all $f$ in the domain of $Y$, and this, together with the uniform boundedness of the $F_n(\Delta) Y$, implies that $\overline{F_n(\Delta) Y} \to \overline{\Delta^{-1/2} Y}$ in the strong operator topology; this proves part \ref{en:rieszdef_firstappr}.

Finally, if $(G_n)_{n \in \NN}$ is as in part \ref{en:rieszdef_secondappr}, then the sequence $(\sqrt{G_n})_{n \in \NN}$ satisfies the assumptions of part \ref{en:rieszdef_firstappr}, so the range of $G_n(\Delta)^{1/2}$ is contained in the domain of $Y$, and moreover $Y G_n(\Delta)^{1/2} \to \overline{Y \Delta^{-1/2}}$ and $\overline{G_n(\Delta)^{1/2} Z} \to \overline{\Delta^{-1/2} Z}$ in the strong operator topology. From this it follows that the range of $G_n(\Delta) Z = G_n(\Delta)^{1/2} G_n(\Delta)^{1/2} Z$ is contained in the domain of $Y$; moreover, for all $f$ in the domain of $Z$,
\[
Y G_n(\Delta) Z f = Y G_n(\Delta)^{1/2} G_n(\Delta)^{1/2} Z f \to \overline{Y \Delta^{-1/2}} \, \overline{\Delta^{-1/2} Z} f,
\]
and part \ref{en:rieszdef_secondappr} follows.
\end{proof}

We can now define the first-order Riesz transforms we are interested in as the $L^2$-bounded operators $R_Y = \overline{Y \Delta^{-1/2}}$, for all horizontal left-invariant vector fields $Y$. Similarly we define the second-order Riesz transforms $R_{YZ} = -R_Y R_Z^* = \overline{Y \Delta^{-1/2}} \, \overline{\Delta^{-1/2} Z}$ for all horizontal left-invariant vector fields $Y,Z$. By a slight abuse of notation, we write $Y \Delta^{-1/2}$ and $Y \Delta^{-1} Z$ instead of $R_Y$ and $R_{YZ}$.

\subsection{Proof of Theorem \ref{thm:main}}

We are finally able to prove our main result.

\begin{proof}[Proof of Theorem \ref{thm:main}\ref{en:main_first}]
Let $Y$ be any horizontal left-invariant vector field. By Proposition \ref{prp:def_riesz}, $Y \Delta^{-1/2}$ is $L^2$-bounded and
\[
\Gamma(1/2) \, Y \Delta^{-1/2} 
 = \sum_{n\in\ZZ} \int_{2^n}^{2^{n+1}}t^{-1/2} Y e^{-t\Delta} \di t = \sum_{n\in \ZZ} T_n,
\]
where the series converges in the strong $L^2$ operator topology.

Let $k_n$ denote the convolution kernel of the operator $T_n$, which is given by
\[
k_n = \int_{2^n}^{2^{n+1}}t^{-1/2} Y h_t\di t.
\]
For every $n\in\ZZ$,
\[\begin{split}
\nabla_H k_n^* &= \int_{2^n}^{2^{n+1}}t^{-1/2} \nabla_H (Y h_t)^* \di t  \\
&= \int_{2^n}^{2^{n+1}} t^{-1/2} \nabla_H (Y (h_{t/2} * h_{t/2}))^* \di t \\
&= \int_{2^n}^{2^{n+1}} t^{-1/2} (Y h_{t/2})^* * \nabla_H h_{t/2} \di t,
\end{split}\]
where the fact that $h_t = h_{t/2} * h_{t/2}$ and $h_{t/2}^* = h_{t/2}$ was used. Hence
\[
\begin{split}
\| |\nabla_H k_n^*|_g \|_1
& \leq \int_{2^n}^{2^{n+1}} t^{-1/2} \,  \| |\nabla_H h_{t/2}|_g \|_1 \,\| Y h_{t/2}\|_1 \di t  \\
&\lesssim  \int_{2^n}^{2^{n+1}} t^{-3/2} \di t \lesssim 2^{-n/2},
\end{split}
\]
where we have applied Proposition \ref{prp:gradientestimates}.
Moreover, for every $n \in \ZZ$,
\[\begin{split}
\int_G |k_n(x ) |\, \big(1+2^{-n/2} |x|_\dist \big) \di\mu(x)
&\leq \int_{2^n}^{2^{n+1}} t^{-1/2} \|Y h_t  \,  \big(1+2^{-n/2}  |\cdot|_{\dist}  \big) \|_1    \di t \\
&\lesssim \int_{2^n}^{2^{n+1}} t^{-1/2}  \|Y h_t   \exp( t^{-1/2}   |\cdot|_{ \dist}    ) \|_1      \di t \\
&\lesssim  \int_{2^n}^{2^{n+1}} t^{-1}  \,    \di t\lesssim 1 .
\end{split}\]
by Proposition \ref{prp:gradientestimates}. The required boundedness properties of $Y \Delta^{-1/2}$ then follow from Theorem \ref{thm:Teolim} and Remark \ref{rem:Teolim_grad}.
\end{proof}

\begin{proof}[Proof of Theorem \ref{thm:main}\ref{en:main_second}]
Notice that by linearity it suffices to prove the result when $Y=\vfG_j$ and $Z=\vfG_l$, for every $j,l\in\{0,\dots,q\}$.

By Proposition \ref{prp:def_riesz}, the operator $\vfG_j \Delta^{-1} \vfG_l$ is $L^2$-bounded and its adjoint is $\vfG_l \Delta^{-1} \vfG_j$. Hence it is enough to prove that $\vfG_j \Delta^{-1} \vfG_l$ is of weak type $(1,1)$, bounded on $L^p(G)$ for $p \in (1,2]$, and bounded from $H^1(G)$ to $L^1(G)$. Moreover, again by Proposition \ref{prp:def_riesz},
\[
-\vfG_j \Delta^{-1} \vfG_l 
= \int_1^\infty \vfG_j (\vfG_l e^{-t\Delta})^* \di t + \sum_{n < 0} \int_{2^{n}}^{2^{n+1}} \vfG_j (\vfG_l e^{-t\Delta})^* \di t
=T^{(\infty)} + \sum_{n < 0} T_n,
\]
where $\int_1^\infty = \lim_{R \to \infty} \int_1^R$ and both this limit and the convergence of the series are meant in the strong $L^2$ operator topology.

Let $k_n$ denote the convolution kernel of the operator $T_n$, which is given by
\[
k_n= \int_{2^{n}}^{2^{n+1}} \vfG_j (\vfG_l h_t)^* \di t = \int_{2^{n}}^{2^{n+1}} (\vfG_l h_{t/2})^* * \vfG_j h_{t/2} \di t,
\]
and note that
\[
\nabla_H k_n^* = \int_{2^{n}}^{2^{n+1}} (\vfG_j h_{t/2})^* * \nabla_H \vfG_l h_{t/2} \di t.
\]
We then have 
\[
\| |\nabla_H k_n^*|_g \|_1 \leq \int_{2^{n}}^{2^{n+1}} \| \vfG_j h_{t/2} \|_1  \| |\nabla_H \vfG_l h_{t/2}|_g \|_1 \di t \lesssim \int_{2^{n}}^{2^{n+1}} t^{-3/2} \di t \lesssim 2^{-n/2},
\]
by Proposition \ref{prp:small_time_gaussian_heat} (note that $2^{n+1} \leq 1$),
while
\[\begin{split}
&\int_G |k_n(x)| \,(1+2^{-n/2}|x|_\dist) \di\mu(x) \\
&\leq \int_{2^{n}}^{2^{n+1}} \| \vfG_j h_{t/2} \,(1+2^{-n/2}|\cdot|_\dist) \|_1 \, \| \vfG_l h_{t/2} \,(1+2^{-n/2}|\cdot|_\dist) \|_1 \di t \\
&\lesssim \int_{2^{n}}^{2^{n+1}} \| \vfG_j h_{t/2} \,\exp(t^{-1/2}|\cdot|_\dist) \|_1 \, \| \vfG_l h_{t/2} \,\exp(t^{-1/2}|\cdot|_\dist) \|_1 \di t \\
&\lesssim \int_{2^{n}}^{2^{n+1}} t^{-1} \di t \lesssim 1,
\end{split}\]
again by Proposition \ref{prp:small_time_gaussian_heat}. By Theorem \ref{thm:Teolim} and Remark \ref{rem:Teolim_grad}, this proves that $\sum_{n<0}^\infty T_n$
is of weak-type $(1,1)$, bounded on $L^p(G)$ for $p \in (1,2]$ and bounded from $H^1(G)$ to $L^1(G)$. 

In order to conclude, it is enough to show that the convolution kernel $k^{(\infty)}$ of $T^{(\infty)}$ is in $L^1(G)$.
Note that, similarly as above,
\begin{equation}\label{eq:kinfty}
k^{(\infty)} = \int_{1}^{\infty} \vfG_j (\vfG_l h_t)^* \di t .
\end{equation}
Thanks to Proposition \ref{prp:pointwise_heat}, the above integral is absolutely convergent and the convergence is uniform on compact subsets of $G$.

To prove that $k^{(\infty)}$ is in $L^1(G)$ we need to distinguish different cases.

\smallskip

\emph{I. Case $j,l\in\{1,\dots,q\}$.}
By Proposition \ref{prp:gradientestimates_mixed},
\[
\| k^{(\infty)}  \|_1 \leq \int_{1}^{\infty} \| \vfG_j (\vfG_l h_t)^*\|_1 \di t \lesssim \int_{1}^{\infty} t^{-3/2} \di t \lesssim 1.
\]
 
\smallskip

\emph{II. Case $j\in\{1,\dots,\fdim\}$, $l=0$.}
Recall from \eqref{eq:XjX0_dec} the decomposition
\[
\vfG_j (\vfG_0 h_t)^*(z,u) = I_1 + I_2.
\]
By \eqref{eq:kinfty} it is enough to prove that $\int_1^\infty I_1 \di t \in L^1(G)$ and $\int_1^\infty I_2 \di t \in L^1(G)$. 

The estimate $\| I_1\|_1 \lesssim t^{-3/2}$ for $t \geq 1$ is obtained analogously as in the proof of Proposition \ref{prp:gradientestimates_mixed} (here Lemma \ref{lem:innerintegral} with $\alpha=1/2$ and $\beta=3/2$ is used), whence $\int_1^\infty I_1 \di t \in L^1(G)$.
As for the remaining summand,
\[
\int_1^\infty I_2 \di t = -\int_0^\infty \int_1^\infty \partial_\xi [\xi \, \Psi_t(\xi)] \di t \, \exp\left(-\frac{\cosh u}{\xi}\right) \, e^u \vfN_j h^N_{e^u \xi /2}(z) \di\xi
\]
(the integral in $\xi$ and $t$ is absolutely convergent for fixed $z$ and $u$ due to \eqref{eq:I2_absint}, thus changing the order of integration is justified by Fubini's Theorem),
whence, by Lemma \ref{lem:cancellation}\ref{en:cancellation_derintest} and Proposition \ref{prp:n_heat_estimates},
\begin{equation}\label{eq:infinity_integrability_cancellation}
\begin{split}
\left\| \int_1^\infty I_2 \di t \right\|_1 
&\lesssim \int_0^\infty \cosh \theta \exp\left(-\frac{\theta^2}{4}\right) \\
&\quad\times \int_\RR \int_0^\infty \frac{e^{u/2}}{\xi^{2+1/2}}  \exp\left(- \frac{\cosh \theta + \cosh u}{\xi}\right) \di\xi \di u \di \theta \\
&\lesssim \int_0^\infty \frac{\cosh \theta}{e^\theta} \exp\left(-\frac{\theta^2}{4}\right) \di \theta \lesssim 1
\end{split}
\end{equation}
(here Lemma \ref{lem:innerintegral} was applied with $\alpha=\beta=1/2$).
 
\smallskip
 
\emph{III. Case $l\in\{1,\dots,\fdim\}$, $j=0$.}
Notice that $(k^{(\infty)})^*= \int_{1}^{\infty} \vfG_l (\vfG_0 h_t)^* \di t$ and we proved above that $ \int_{1}^{\infty} \vfG_l (\vfG_0 h_t)^* \di t $ is in $L^1(G)$. Thus $k^{(\infty)}$ is integrable also in this case. 

\smallskip
 
\emph{IV. Case $l=j=0$.}
Recall from \eqref{eq:X0X0_dec} the decomposition
\[
\vfG_0 (\vfG_0 h_t)^*(z,u) = J_1+2J_2+J_3.
\]
Note now that, by \eqref{eq:defPsi} and Proposition \ref{prp:n_heat_estimates}, for $t \geq 1$,
\[\begin{split}
\|J_1 \|_1 &\lesssim \int_\RR \int_0^\infty |\Psi_t(\xi)| \, \frac{1}{\xi^2} \, \exp\left(-\frac{\cosh u}{\xi}\right) \di \xi \di u \\
&\lesssim t^{-1/2} \int_0^\infty \exp\left(-\frac{\theta^2}{4t}\right) \sinh \theta \left| \sin \frac{\pi\theta}{2t}\right| \\
&\qquad \times \int_\RR \int_0^\infty \frac{1}{\xi^{2+2}} \, \exp\left(-\frac{\cosh \theta + \cosh u}{\xi}\right) \di\xi \di u \di\theta \\
&\lesssim t^{-1/2} \int_0^\infty \frac{\sinh \theta}{e^\theta} e^{-2\theta} (1+\theta) \frac{\theta}{t} \di\theta \lesssim t^{-3/2}
\end{split}\]
(here Lemma \ref{lem:innerintegral} with $\alpha=0$ and $\beta=2$ was applied), whence $\int_1^\infty J_1 \di t \in L^1(G)$.

As for $J_2$, note that
\[
\int_1^\infty J_2 \di t = \int_0^\infty \int_1^\infty \partial_\xi [\xi \, \Psi_t(\xi)] \di t \, \frac{\cosh u}{\xi} \exp\left(-\frac{\cosh u}{\xi}\right) \, h^N_{e^u \xi /2}(z) \di \xi
\]
(here the integral in $\xi$ and $t$ is absolutely convergent due to \eqref{eq:J123_absint}, so changing the order of integration is justified).
Hence, similarly as in \eqref{eq:infinity_integrability_cancellation}, one can apply Lemma \ref{lem:cancellation}\ref{en:cancellation_derintest} and Proposition \ref{prp:n_heat_estimates}, as well as Lemma \ref{lem:innerintegral} with $\alpha=\beta=1$, to obtain that $\int_1^\infty J_2 \di t \in L^1(G)$.

Finally, 
\[
\int_1^\infty J_3 \di t = \int_0^\infty \int_1^\infty \partial_\xi [\xi \partial_\xi [\xi \, \Psi_t(\xi)] ] \di t \, \exp\left(-\frac{\cosh u}{\xi}\right) \, h^N_{e^u \xi /2}(z) \di \xi
\]
(again, the integral in $\xi$ and $t$ is absolutely convergent due to \eqref{eq:J123_absint}).
Now, from Lemma \ref{lem:cancellation}\ref{en:cancellation_derint} it is immediately obtained that
\begin{multline*}
\int_1^\infty \partial_\xi [\xi \partial_\xi [\xi \Psi_t(\xi)]] \di t = \partial_\xi \left[\xi \int_1^\infty \partial_\xi [\xi \Psi_t(\xi)] \di t \right] \\
= \frac{\xi^{-2}}{\sqrt{4\pi^3}} \int_0^\infty \int_0^\pi \cos \frac{s\theta}{2} \exp\left(\frac{s^2 -\theta^2}{4}-\frac{\cosh \theta}{\xi}\right) \cosh\theta \left[ \frac{\cosh \theta}{\xi} - 1 \right] \di s \di\theta
\end{multline*}
and
\[\begin{split}
\left|\int_1^\infty \partial_\xi [\xi \partial_\xi [\xi \Psi_t(\xi)]] \di t\right| 
&\lesssim  \xi^{-2} \int_0^\infty \exp\left(-\frac{\theta^2}{4}-\frac{\cosh \theta}{\xi}\right) \cosh\theta \di\theta \\
&\qquad +\xi^{-3} \int_0^\infty \exp\left(-\frac{\theta^2}{4}-\frac{\cosh \theta}{\xi}\right) \cosh^2 \theta \di\theta \\
&= S' + S''.
\end{split}\]
Correspondingly
\[\begin{split}
\left| \int_1^\infty J_3 \di t \right| &\leq \int_0^\infty \left| \int_1^\infty \partial_\xi [\xi \partial_\xi [\xi \, \Psi_t(\xi)] ] \di t \right| \, \exp\left(-\frac{\cosh u}{\xi}\right) \, h^N_{e^u \xi /2}(z) \di \xi \\
&\lesssim \int_0^\infty S' \, \exp\left(-\frac{\cosh u}{\xi}\right) \, h^N_{e^u \xi /2}(z) \di \xi \\
&\qquad + \int_0^\infty S'' \, \exp\left(-\frac{\cosh u}{\xi}\right) \, h^N_{e^u \xi /2}(z) \di \xi \\
&= J' + J''.
\end{split}\]
The fact that $J' \in L^1(G)$ is then proved similarly as in \eqref{eq:infinity_integrability_cancellation}, by using Proposition \ref{prp:n_heat_estimates} and Lemma \ref{lem:innerintegral} with $\alpha=\beta=0$. As for $J''$, a similar argument gives
\[\begin{split}
\| J'' \|_1 
&\lesssim \int_0^\infty \cosh^2 \theta \exp\left(-\frac{\theta^2}{4}\right) \\
&\quad\times \int_\RR \int_0^\infty \frac{1}{\xi^{2+1}}  \exp\left(- \frac{\cosh \theta + \cosh u}{\xi}\right) \di\xi \di u \di \theta \\
&\lesssim \int_0^\infty \frac{\cosh^2 \theta}{e^{2\theta}} \exp\left(-\frac{\theta^2}{4}\right) \di \theta \lesssim 1
\end{split}\]
(here Lemma \ref{lem:innerintegral} was applied with $\alpha=0$ and $\beta=1$), and we are done.
\end{proof}
 
\begin{rem}
In view of the estimates of Propositions \ref{prp:gradientestimates_mixed} and \ref{prp:gradientestimates_mixedmixed}, one can prove that, in the case $j,l \in \{1,\dots,\fdim\}$, the operator $T^{(\infty)}$ in the above proof satisfies the size and smoothness assumptions \eqref{eq:czwl1} and \eqref{eq:czhoelder} of Theorem \ref{thm:Teolim} too.
\end{rem}

\section*{Acknowledgements}
The anonymous referee is gratefully thanked for a number
of suggestions that helped to improve the presentation of this work.

The authors are members of the Gruppo Nazionale per l'Analisi Mate\-ma\-tica, la Probabilit\`a e le loro Applicazioni (GNAMPA) of the Istituto Na\-zio\-nale di Alta Matematica (INdAM).
This work was partially supported by the EPSRC Grant ``Sub-Elliptic Harmonic Analysis'' (EP/P002447/1), the Progetto GNAMPA 2017 ``Analisi armonica e teoria spettrale di Laplaciani'' and the Progetto PRIN 2015 ``Variet\`a reali e complesse: geometria, topologia e analisi armonica''.

\end{document}